\documentclass[12pt,a4paper,psamsfonts]{amsart}

\usepackage{fancyhdr}
\usepackage{appendix}
\usepackage{amssymb,amscd,amsxtra,calc}
\usepackage{mathrsfs}
\usepackage{cmmib57}
\usepackage{multirow}
\usepackage[all]{xy}
\usepackage{longtable}
\usepackage[colorlinks,linkcolor=red,anchorcolor=blue,citecolor=blue]{hyperref}

\theoremstyle{plain}
    \newtheorem{thm}{Theorem}[section]

    \newtheorem{lemma}[thm]{Lemma}
    \newtheorem{proposition}[thm]{Proposition}
    \newtheorem{question}[thm]{Question}
    \newtheorem{theorem}[thm]{Theorem}

\theoremstyle{definition}

    \newtheorem{notation}[thm]{Notation}
    \newtheorem*{notation*}{Notation and Terminology}
    \newtheorem{remark}[thm]{Remark}
    \newtheorem*{ack}{Acknowledgments}
\theoremstyle{remark}

\newcommand{\Q}{\mathbb{Q}}

\newcommand{\R}{\mathbb{R}}

\newcommand{\Z}{\mathbb{Z}}

\newcommand{\diag}{\operatorname{diag}}

\newcommand{\NE}{\overline{\operatorname{NE}}}
\newcommand{\Nef}{\operatorname{Nef}}

\newcommand{\NS}{\operatorname{NS}}

\newcommand{\PE}{\operatorname{PE}}

\newcommand{\N}{\operatorname{N}}

\newcommand{\Pic}{\operatorname{Pic}}

\begin{document}

\title[Dynamical rigidity]
{Rigidity of rationally connected smooth projective varieties from dynamical viewpoints}

\address{
School of Mathematical Sciences,  Key Laboratory of MEA (Ministry of Education) Shanghai Key Laboratory of PMMP   East China Normal University, Shanghai 200241, China; and School of Mathematics, Korea Institute For Advanced Study\\ 
85 Hoegiro, Dongdaemungu Seoul 02455, Republic of Korea}
\email{smeng@math.ecnu.edu.cn}

\author{Sheng Meng, Guolei Zhong}

\address
{National University of Singapore,
Singapore 119076, Republic of Singapore
}
\email{zhongguolei@u.nus.edu}
\begin{abstract}
Let $X$ be a rationally connected smooth projective variety of dimension $n$.
We show that $X$ is a toric variety if and only if $X$ admits an int-amplified endomorphism with totally invariant ramification divisor.
We also show that $X\cong (\mathbb{P}^1)^{\times n}$
if and only if $X$ admits a surjective endomorphism $f$ such that the eigenvalues of $f^*|_{\textup{N}^1(X)}$ (without counting multiplicities) are $n$ distinct real numbers  greater than $1$.
\end{abstract}



\maketitle
\tableofcontents

\section{Introduction}

We work over an algebraically closed field $k$ of characteristic $0$.
Let $X$ be a smooth projective variety which is {\it rationally connected}, i.e., any two general points of $X$ can be connected by a chain of rational curves; see \cite{Cam92} and \cite{KMM92}.
A natural interest is to characterize such varieties in terms of some dynamic assumptions.

Let $f:X\to X$ be a surjective endomorphism.
Then $f^*$ induces an invertible linear map on $\N^1(X):=\NS(X)\otimes_{\Z} \R$ where $\NS(X)$ is the N\'eron-Severi group of $X$.
In this paper, we focus on the case when all the eigenvalues of $f^*|_{\N^1(X)}$ have modulus greater than 1.
Such $f$ is also known to be {\it int-amplified}, i.e., $f^*L-L$ is ample for some ample Cartier divisor $L$; see \cite[Theorem 1.1]{Men20}. 
It is a natural generalization of the {\it $q$-polarized} endomorphism, i.e., $f^*H\sim qH$ for some ample Cartier divisor $H$ and integer $q>1$.

To the best knowledge of the authors, 
the int-amplified assumption is necessary for one to get restrictions on the rigidity of $X$.
Roughly speaking, we want to exclude the effects of automorphisms.
On the one hand, given a non-isomorphic surjective endomorphism $f$ on a variety, it is easy to obtain a new non-isomorphic surjective endomorphism by taking the product  of $f$ and an automorphism of another arbitrary variety,  in which case, one can get little information on this new variety.
In general, more complicated situations rather than the product case may occur; see \cite[Example 10.2]{Men20}.
On the other hand, we are concerned about the periodic points of such non-isomorphic endomorphisms. 
Fakhruddin showed in \cite[Theorem 5.1]{Fak03} that \textit{amplified} endomorphism (i.e., $f^*L-L$ is ample for not necessarily ample $L$) has countable and Zariski dense periodic points.
This property is also called \textit{PCD} (over an uncountable field of characteristic 0) as studied by the first author in \cite{Men23}. 
However, there exist amplified (and hence PCD) automorphisms which can be easily constructed on the  abelian varieties of product type (cf.\,\cite[Theorem 6.2, Example 6.6]{Men23}).

In \cite[Theorem 1.10]{Men20}, 
the first author proved the equivariant minimal model program (MMP) for int-amplified endomorphisms (cf.\,\cite[Definition 2.1]{MZ20a}), 
generalizing an early result for polarized endomorphisms by Zhang and the first author (cf.\,\cite[Theorem 1.8]{MZ18}).
We refer to \cite{MZ18}, \cite{CMZ20}, \cite{Men20}, \cite{MZ20a}, \cite{MZ20b}, \cite{Zho21} for details and further generalizations about equivariant MMP.
In this way, 
Yoshikawa further proved the following result, 
answering partially a conjecture of Broustet and Gongyo (cf.\,\cite[Conjecture 1.2]{BG17}).
It is also the initial point of this paper.

\begin{theorem}[{\cite[Corollary 1.4]{Yos21}}]\label{main-thm-y}
A rationally connected smooth projective variety $X$ is of Fano type if it admits an int-amplified endomorphism  (cf.\,Notation \ref{setup2.1}).
\end{theorem}

When $X$ is a rationally connected smooth projective surface admitting a non-isomorphic endomorphism, 
it was proved by Nakayama \cite[Theorem 3]{Nak02} that $X$ is then a toric surface, 
answering affirmatively a conjecture proposed by Sato (cf.\,\cite[Conjecture 2]{Nak02}).
Note that every toric variety is of Fano type (cf.\,Notation \ref{setup2.1}); however, 
there exist many non-rational Fano varieties.
Based on Theorem \ref{main-thm-y}, 
we ask the following question which is a higher dimensional analogue of Sato's conjecture; see also \cite[Question 4.4]{Fak03} for the polarized case.
We note that Question \ref{main-que-toric} is known for Fano threefolds (see\,\cite[Theorem 1.4]{MZZ22}) and for Fano fourfolds admitting a conic bundle structure (see\,\cite[Theorem 1.4]{JZ23}).

\begin{question}\label{main-que-toric}
Let $X$ be a rationally connected smooth projective variety admitting an int-amplified endomorphism.
Is $X$ a toric variety?
\end{question}

\begin{remark}[{\textbf{Motivation and Difficulties for Question \ref{main-que-toric}}}]
We note that toric varieties usually have lots of dynamically interesting symmetries (cf.\,\cite{Nak02}), and our Question \ref{main-que-toric} here is sort of a converse direction to it.
In general, given a non-isomorphic surjective endomorphism  on a rationally connected smooth projective variety (or even a Fano manifold) $X$, 
it is very difficult to find a big torus on $X$; hence
a positive answer to Question \ref{main-que-toric} reveals a very deep symmetric essence shared by toric varieties and int-amplified endomorphisms.	
In this paper, 
we shall show two situations from the aspects of geometry and cohomology in which Question \ref{main-que-toric} holds.
\end{remark}

The following question proposed by Zhang and the first author aimed to generalize the results for polarized endomorphisms in \cite[Theorem 2.1]{HN11} and \cite[Corollary 1.4]{MZ19} to the int-amplified case.
\begin{question}[{\cite[Question 10.1]{MZ22}}]\label{Ques_MZ_19}
Let $f:X\to X$ be an int-amplified endomorphism of a rationally connected smooth projective variety.
Suppose there is an $f^{-1}$-invariant reduced divisor $D$ such that $f|_{X\backslash D}:X\backslash D\to X\backslash D$ is \'etale.
Is $(X,D)$ a toric pair? 
\end{question}

However, due to a gap of the slope semistability, Zhang and the first author can only deal with the case when $f^*|_{\N^1(X)}$ has at most two eigenvalues in \cite[Theorem 10.6]{MZ22}. 
We strongly recommend \cite[Section 10]{MZ22} for a detailed explanation.
In this paper, we will mainly focus on overcoming this gap and answer Question \ref{Ques_MZ_19} affirmatively. 

\begin{theorem}\label{thm_int_toric}
Let $X$ be a rationally connected smooth projective variety with $D$ a reduced divisor.
Then $(X, D)$ is a toric pair if and only if
$X$ admits an int-amplified endomorphism $f$ such that $f|_{X\backslash D}:X\backslash D\to X\backslash D$ is \'etale.
\end{theorem}

In \cite[Theorem 1.11]{Men20}, replacing $f$ by a suitable power, $f^*|_{\N^1(X)}$ can be viewed as a diagonal matrix $\diag[\lambda_1,\cdots, \lambda_{\rho}]$, 
where $\rho:=\rho(X)=\dim_{\R}\N^1(X)$ and $\lambda_i$ are (possibly the same) integers greater than $1$. 
Here, $\rho$ can be arbitrarily large in general, even if $n=\dim(X)$ is fixed.
Nevertheless, the number $r$ of eigenvalues of $f^*|_{\N^1(X)}$ (without counting multiplicities) is bounded by $n$; see Proposition \ref{prop-Di}.
Note that if $r=1$, then $f$ is the usual polarized endomorphism.
For another extremal case when $r=n$, we show the strongest splitting rigidity of $X$.
The key idea is to show that the equivariant MMP of $X$ involves only with (conic bundle type) Fano contractions  of smooth Fano varieties, and then we are able to apply the adjunction formula in the most comfortable way and prove by induction on  the dimension of $X$.

\begin{theorem}\label{thm_rc_product}
Let $X$ be a rationally connected smooth projective variety of dimension $n$.
Then
$$X\cong (\mathbb{P}^1)^{\times n}$$
if and only if $X$ admits a surjective endomorphism $f$ such that  the eigenvalues of $f^*|_{\N^1(X)}$ (without counting multiplicities) are $n$ distinct real numbers greater than 1.
\end{theorem}

\begin{remark}
In the proof of Theorem \ref{thm_rc_product}, we get a finite surjective morphism $$\psi: X\to  (\mathbb{P}^1)^{\times n}$$ 
with $\rho(X)=n$ and all the effective divisors of $X$ being nef.
The remaining main difficulty is to show that $\psi$ is an isomorphism or simply $X\cong (\mathbb{P}^1)^{\times n}$.
When $n=\rho(X)=2$, we have $X\cong \mathbb{P}^1\times \mathbb{P}^1$ by an easy observation of smooth Fano surfaces. 
However, for the higher dimensional cases, this is in general not true without dynamical concerns.
For example, when $n=\rho(X)=3$, there exists a double cover $\psi$ whose branch locus is a divisor of tridegree $(2,2,2)$ (cf.\,\cite[Table 3]{MM81}).
Such $X$ admits three different Fano contractions to $\mathbb{P}^1\times \mathbb{P}^1$ which are  conic bundles with non-empty discriminant locus (i.e., the contraction morphisms are not smooth).
In particular, we have $X\not\cong \mathbb{P}^1\times \mathbb{P}^1\times \mathbb{P}^1$.
\end{remark}

At the end of this section, we propose the following splitting question for the general $1\le r\le n$ which reduces Question \ref{main-que-toric} to the polarized case. Question \ref{main-que-split} is not true in the general settings of smooth projective varieties; see examples in \cite[Section 10]{Men20}.
\begin{question}\label{main-que-split}
Let $f:X\to X$ be a surjective endomorphism of a rationally connected smooth projective variety such that the eigenvalues of $f^*|_{\N^1(X)}$ (without counting multiplicities) are distinct integers $\lambda_1,\cdots, \lambda_{r}$ greater than 1.
Will we have
$$X\cong X_1\times \cdots\times X_{r}$$
where each $X_i$ is a smooth projective variety of Fano type and $f$ splits to $f|_{X_i}$ which is $\lambda_i$-polarized?
\end{question}

\begin{ack}
The authors would like to thank Professor De-Qi Zhang  for many inspiring discussions, and  the anonymous referee  for suggestions to improve the paper. The first author is supported by Science and Technology Commission of Shanghai Municipality (No. 22DZ2229014), a National Natural Science Fund, and a Research Fellowship of KIAS. The second author is supported by a President's Graduate Scholarship of NUS. 
\end{ack}

\section{Preliminary}
\begin{notation}\label{setup2.1}
Let $X$ be a projective variety. 
We use the following notation throughout this paper. 
\begin{longtable}{p{1cm} p{0.5cm} p{12cm}}
$\textup{Pic}(X)$ && the group of Cartier divisors of $X$ modulo linear equivalence $\sim$\\
$\textup{Pic}^\circ(X)$ && the neutral connected component of $\Pic(X)$\\
$\textup{NS}(X)$ & & $\textup{Pic}(X)/\textup{Pic}^\circ(X)$, the N\'eron-Severi group of $X$\\
$\equiv$&& the numerical equivalence for $\mathbb{R}$-Cartier divisors\\
$\equiv_w$ && the weak numerical equivalence for $r$-cycles. 
Two $r$-cycles $C_1$ and $C_2$ are said to be \textit{weakly numerically equivalent} (denoted by $C_1\equiv_w C_2$) if $(C_1-C_2)\cdot L_1\cdots L_r=0$ for all Cartier divisors $L_i$ on $X$; cf.\,\cite[Definition 2.2]{MZ18} and the references therein. \\
$\N^1(X)$ & & $\textup{NS}(X)\otimes_\mathbb{Z}\mathbb{R}$, the space of $\mathbb{R}$-Cartier divisors modulo numerical equivalence $\equiv$\\
$\N_r(X)$ & & the space of $r$-cycles modulo weak numerical equivalence $\equiv_w$\\
$\rho(X)$ && $\dim_{\mathbb{R}}\N^1(X)$, the Picard number of $X$\\
$\kappa(X,D)$  && the Iitaka dimension of a $\mathbb{Q}$-Cartier divisor $D$\\
 $\textup{Nef}(X)$&& the cone of nef classes in $\N^1(X)$\\
 $\textup{PE}^1(X)$&& the cone of pseudo-effective classes in $\N^1(X)$\\
 $\overline{\textup{NE}}(X)$&& the cone of pseudo-effective classes in $\N_1(X)$\\
\end{longtable}
\vspace{-1em}
\begin{itemize}
\item The above cones  are $(f^*)^{\pm 1}$-invariant for any surjective endomorphism $f:X\to X$.
\item A surjective endomorphism $f:X\to X$ is \textit{$q$-polarized} if $f^*H\sim qH$ for some ample Cartier divisor $H$ and integer $q>1$, or equivalently if $f^*B\equiv qB$ for some big $\mathbb{R}$-Cartier divisor $B$ and integer $q>1$ (cf.\,\cite[Proposition 3.6]{MZ18}).
\item A surjective endomorphism $f:X\to X$ is \textit{int-amplified} if $f^*L-L=H$ for some ample Cartier divisors $L$ and $H$, or equivalently if $f^*L-L=H$ for some big $\mathbb{R}$-Cartier divisors $L$ and $H$ (cf.\,\cite[Theorem 1.1]{Men20}).
\item We say that a normal projective variety $X$ is \textit{of Fano type}, if there exists an effective Weil $\mathbb{Q}$-divisor $\Delta$ on $X$ such that the pair $(X,\Delta)$ has at worst klt  singularities and $-(K_X+\Delta)$ is ample (cf.\,\cite[Lemma-Definition 2.6]{PS09}). If $\Delta=0$, we say that $X$ is a \textit{Fano variety}.
\item  A finite surjective morphism is \textit{quasi-\'etale} if it is \'etale in codimension one.
\item We say that a normal variety $X$ is a \textit{toric variety} if $X$ contains an algebraic torus $T=(k^*)^n$ as an (affine) open dense subset such that the natural multiplication action of $T$ on itself extends to an action on the whole variety. In this case, let $D:=X\backslash T$, which is a divisor; the pair $(X,D)$ is said to be a \textit{toric pair}.
\item Let $(X,\Delta)$ be a log pair. Write $\Delta=\sum_i a_iD_i$ with each $a_i>0$ and $D_i$ being distinct irreducible divisors.
Denote by 
$$\left<\Delta\right>:=\lfloor \Delta \rfloor +\lceil 2\Delta \rceil-\lfloor 2\Delta \rfloor=\sum_{i:a_i>1/2}D_i.$$
A \textit{decomposition} of $\Delta$ is an expression of the form $\sum_{i=1}^k a_iS_i\le\Delta$ where $S_i\ge 0$ are $\mathbb{Z}$-divisors and $a_i\ge 0$ for each $i$.
The \textit{complexity} of this decomposition is $n+r-d$, where $r$ is the rank of the vector space spanned by $S_1,S_2,\cdots,S_k$ in the space of Weil $\mathbb{R}$-divisors modulo algebraic equivalence and $d=\sum a_i$.
The \textit{complexity} $c=c(X,\Delta)$ of $(X,\Delta)$ is the infimum of the complexity of any decomposition of $\Delta$ (cf.\,\cite[Definition 1.1]{BMSZ18}).
\end{itemize}
\end{notation}

In what follows, we prepare several preliminary results for the use of our proofs. 
First, the following theorem gives a geometric characterization of toric varieties involving the complexity by Brown, M$^\text{C}$Kernan, Svaldi and Zong.

\begin{theorem}[{\cite[Theorem 1.2]{BMSZ18}}]\label{thm-BMSZ}
Let $X$ be a proper variety of dimension $n$ and let $(X,\Delta)$ be a log canonical pair such that $-(K_X+\Delta)$ is nef.
If $\sum a_iS_i$ is a decomposition of complexity $c$ less than one, then there is a divisor $D$ such that $(X,D)$ is a toric pair, where $D\ge\left<\Delta\right>$ and all but $\lfloor 2c \rfloor$ components of $D$ are elements of the set $\{S_i~|~1\le i\le k\}$.
\end{theorem}

We give  a simple version of Theorem \ref{thm-BMSZ}, which is enough for our application.
\begin{theorem}[{\cite[Remark 4.4 (1)]{MZ19}}]\label{thm-BMSZ2}
Let $X$ be a smooth projective variety of dimension $n$ and let $D=\sum_{i=1}^d D_i$ be a reduced divisor such that $(X,D)$ is a log canonical pair and $K_X+D\equiv 0$.
Suppose the complexity $c(X,D)\le 0$.
Then $(X,D)$ is a toric pair.
\end{theorem}

The following result is well-known due to the cone theorem.
\begin{lemma}[{cf.\,\cite[Theorem 3.7]{KM98}, \cite[Corollary 1.3.2]{BCHM10}}]\label{lem-mds}
Let $X$ be a normal projective variety of Fano type.
Then $\Nef(X)$ is generated by finitely many base point free (extremal) Cartier divisors.
\end{lemma}

\begin{lemma}\label{lem-minimal}
Let $X$ be a  projective variety with $\textup{Nef}(X)=\textup{PE}^1(X)$. Then any generically finite surjective morphism $\pi:X\to Y$ to a  projective variety $Y$ is finite. 
\end{lemma}
\begin{proof}
Let $\pi:X\to Y$ be a generically finite surjective morphism. 
Fixing an ample Cartier divisor $H$ on $Y$, 
we have that $\pi^*H$ is big and thus ample by assumption. 
So $\pi$ does not contract any curve and hence $\pi$ is finite. 
\end{proof}

\begin{lemma}\label{lem_positive_notreplaced}
Let $f:X\to X$ be a surjective endomorphism of a normal projective variety such that all the eigenvalues of $f^*|_{\N^1(X)}$ are positive real numbers. Suppose that $(f^m)^*D\equiv\lambda^m D$ for some $\lambda>0$ and integer $m\ge 1$. Then $f^*D\equiv\lambda D$.
\end{lemma}
\begin{proof}
We may assume $D\not\equiv 0$ and $m\ge 2$. 
Let $\varphi:=f^*|_{\N^1(X)}$ and we regard $\varphi$ as a matrix. 
Note that $D\in\ker (\varphi^m-\lambda^m)$ and
 $$\varphi^m-\lambda^m=\prod_{i=0}^{m-1}(\varphi-\lambda \cdot \xi_m^i),$$
where $\xi_m$ is a primitive $m$-th root of unity. 
Suppose that  $D\not\in\ker (\varphi-\lambda)$. 
Then there exists $1\le j\le m-1$ such that
$$\widetilde{D}:=\left(\prod_{i=0}^{j-1}(\varphi-\lambda \cdot \xi_m^i)\right)(D)\not\equiv 0 \,\textup{ and }\,(\varphi-\lambda\cdot \xi_m^j)(\widetilde{D})\equiv 0. $$
Note that $\lambda\cdot\xi_m^j$ is  not a positive real number but an eigenvalue of $\varphi$.
So we get a contradiction.
\end{proof}

\begin{remark}\label{rmk-split}
Let $X:=C_1\times\cdots\times C_n$ with each $C_i\cong\mathbb{P}^1$. 
Let $f:X\to X$ be a surjective endomorphism. 
Then replacing $f$ by a power, we have
$$f=f_1\times\cdots\times f_n$$
where each $f_i: C_i\to C_i$ is a surjective endomorphism.
Moreover, if all the eigenvalues of $f^*|_{\N^1(X)}$ are already positive real numbers, then we don't need to replace $f$ by a power above by Lemma \ref{lem_positive_notreplaced}.	
Note that this kind of splitting result holds true when each $C_i$ is a normal projective variety with $H^1(C_i, \mathcal{O}_{C_i})=0$; see \cite[Theorem 4.6]{San20}.
\end{remark}

\section{Proof of Theorem \ref{thm_int_toric}}
In this section, we shall answer Question \ref{Ques_MZ_19} affirmatively. 
Let $X$ be a normal projective variety and $D$ a reduced divisor on $X$.
Let $j:U\hookrightarrow X$  be a smooth open subset of $X$ with $\textup{codim}\,(X\backslash U)\ge 2$ and $D\cap U$ being a normal crossing divisor.
Denote by 
$$\hat{\Omega}_X^1(\textup{log}\,D):=j_*\Omega_U^1(\textup{log}\,D\cap U)$$
where $\Omega_U^1(\textup{log}\,D\cap U)$ is the locally free sheaf of germs of logarithmic $1$-forms over $U$ with poles only along $D\cap U$. 
Note that $\hat{\Omega}_X^1(\textup{log}\,D)$ is a reflexive coherent sheaf on $X$ which is independent of the choice of $U$.

First, with the same notations as above, we recall the following two results  which are borrowed from \cite{MZ19} and \cite{MZ22}, and will be used in the proof of Theorem \ref{thm_int_toric}.
\begin{theorem}[{cf.\,\cite[Theorem 4.5]{MZ19}}]\label{thm-length-free} 
Let $X$ be a normal projective variety of dimension $n$, and $D$ a reduced divisor of $X$.
Then the complexity (cf.\,Notation \ref{setup2.1})
$$c(X,D)\le n+\widetilde{q}(X)-h^0(X,\hat{\Omega}_X^1(\textup{log}\,D))$$
where $\widetilde{q}(X):=q(\widetilde{X})=h^1(\widetilde{X},\mathcal{O}_{\widetilde{X}})$ with $\widetilde{X}$ being a smooth projective model of $X$.
\end{theorem}


\begin{proposition}[{cf.\,\cite[Proposition 10.3]{MZ22}}]\label{pro10.3}
Let $f:X\rightarrow X$ be an int-amplified endomorphism of a normal projective variety $X$ which is  of dimension $n\ge 2$ and smooth in codimension 2. Let $H$ be an ample Cartier divisor  and $D\subseteq X$  a reduced divisor. Suppose that $f^{-1}(D)=D$ and $f|_{X\backslash D}: X\backslash D \rightarrow X\backslash D$ is quasi-\'etale. Then
$$c_1(\hat{\Omega}_X^1(\textup{log}\,D))\cdot H^{n-1}=c_1(\hat{\Omega}_X^1(\textup{log}\,D))^2\cdot H^{n-2}=c_2(\hat{\Omega}_X^1(\textup{log}\,D))\cdot H^{n-2}=0.$$
\end{proposition}

We follow the idea of \cite[Proposition 10.5]{MZ22} to prove the next proposition. 
\begin{proposition}\label{pro10.5}
Let $f:X\rightarrow X$ be an int-amplified endomorphism of a normal projective variety, which is of dimension $n$, of Fano type and smooth in codimension two. Let $H$ be an ample Cartier divisor  and $D\subseteq X$  a reduced divisor. Suppose that $f^{-1}(D)=D$ and $f|_{X\backslash D}: X\backslash D \rightarrow X\backslash D$ is quasi-\'etale. Then	$\hat{\Omega}_X^1(\textup{log}\,D)$ is $H$-slope semistable.
\end{proposition}

\begin{proof}
By Lemma \ref{lem-mds}, $\Nef(X)$ is generated by finitely many nef divisors.
Replacing $f$ by a power, we may assume $f^*$ fixes each extremal ray of $\Nef(X)$.
Then, there exist nef divisors $D_1,\cdots,D_k$ on $X$ such that $f^*D_i\equiv \lambda_iD_i$ and $H\equiv \sum\limits_{i=1}^k a_iD_i$ with each $a_i>0$. 
		
	Suppose the contrary that $\hat{\Omega}_X^1(\textup{log}\,D)$ is not $H$-slope semistable. Let $\mathcal{F}\subseteq \hat{\Omega}_X^1(\textup{log}\,D)$ be the maximal destabilizing subsheaf with respect to $H$ such that 
	$$\mu_H(\mathcal{F}):=\frac{c_1(\mathcal{F})\cdot H^{n-1}}{\textup{rank}\,\mathcal{F}}>\mu_H(\hat{\Omega}_X^1(\textup{log}\,D))=0.$$
Note that the last equality is due to Proposition \ref{pro10.3}. Since $H\equiv\sum a_iD_i$ with each $a_i>0$, there exists a summand  $D_{i_1}\cdots D_{i_{n-1}}$ of $H^{n-1}$ such that
$$c_1(\mathcal{F})\cdot D_{i_1}\cdots D_{i_{n-1}}>0.$$  
Then $ D_{i_1}\cdots D_{i_{n-1}}\not\equiv_w 0$ and there exists a nef Cartier divisor $D_l$ such that
$$f^*D_l=\lambda_l D_l \text{ and } D_{i_1}\cdots D_{i_{n-1}}\cdot D_l>0.$$ 
So
$\deg f=\lambda_{i_1}\cdots \lambda_{i_{n-1}}\cdot \lambda_l$ by the projection formula. 
Since all the $D_{i_j}$ are nef, we have
$$s=\textup{sup}\,\{c_1(\mathcal{F})\cdot D_{i_1}\cdots D_{i_{n-1}}~|~\mathcal{F}\subseteq \hat{\Omega}_X^1(\textup{log}\,D)\}<\infty.$$
Note that $\lambda_l>1$ (cf.\,\cite[Theorem 1.1]{Men20}). 
Then for some $k\gg 1$ and $g:=f^k$, we get the following inequality by the projection formula
$$c_1(g^*\mathcal{F})\cdot D_{i_1}\cdots D_{i_{n-1}}
=\lambda_l^k\cdot c_1(\mathcal{F})\cdot D_{i_1}\cdots D_{i_{n-1}}>s.$$

Let $U$ be a smooth open subset in $X$ such that $\textup{codim}(X\backslash U)\ge 3$ and $D\cap U$ is a normal crossing divisor (cf.\,\cite[Proposition 10.2]{MZ22}). 
Let $j:g^{-1}(U)\hookrightarrow X$ be the inclusion map and $\mathcal{G}:=j_*((g^*\mathcal{F})|_{g^{-1}(U)})$. 
Then we have
$$c_1(\mathcal{G})\cdot D_{i_1}\cdots D_{i_{n-1}}=c_1(g^*\mathcal{F})\cdot D_{i_1}\cdots D_{i_{n-1}}>s.$$ Note that $(g^*\mathcal{F})|_{g^{-1}(U)}\subseteq (g^*\hat{\Omega}_X^1(\textup{log}\,D))|_{g^{-1}(U)}\cong \hat{\Omega}_X^1(\textup{log}\,D)|_{g^{-1}(U)}$, the latter of which is a locally free sheaf.
Since $\textup{codim}(X\backslash g^{-1}(U))\ge 2$ and $j_*$ is left exact, $\mathcal{G}$ is a coherent subsheaf of $\hat{\Omega}_X^1(\textup{log}\,D)$. 
So we get a contradiction.
\end{proof}

\begin{proof}[Proof of Theorem \ref{thm_int_toric}]
Let $n=\dim (X)$. Suppose $(X,D)$ is a toric pair and denote by $T:=X\backslash D\cong (k^*)^n$ the big torus. 
Then the power map 
$$T\to  T \text{ via }  (x_1,\cdots,x_n)\mapsto (x_1^q,\cdots,x_n^q)$$
extends to a surjective endomorphism $f:X\to X$; see \cite[Lemma 4]{Nak02}.
By the construction, the morphism $f$ sends any divisor $D$ to $qD$ via the pull-back; hence $f$ is $q$-polarized.

For another direction, $X$ first is of Fano type by Theorem \ref{main-thm-y}. 
If $\dim (X)=1$, then $X\cong\mathbb{P}^1$ and $D$ is a divisor of two distinct points. 
Assume that $n:=\dim (X)\ge 2$.
By Propositions \ref{pro10.3}, \ref{pro10.5} and \cite[Theorem 1.20]{GKP16}, the reflexive sheaf of germs of logarithmic 1-forms $\hat{\Omega}_X^1(\textup{log}\,D)$ is  free  of rank $n$ since $X$ is simply connected. 
In particular, $h^0(X,\hat{\Omega}_X^1(\textup{log}\,D))=n$.
Now, we compute the complexity $c(X,D)$ of the pair $(X,D)$.
By Theorem \ref{thm-length-free}, 
$$c(X,D)\le n+\widetilde{q}(X)-h^0(X,\hat{\Omega}_X^1(\textup{log}\,D)),$$
where $\widetilde{q}(X)$ is the irregularity of a smooth projective model of $X$. 
Since $X$ is smooth and rationally connected, $\widetilde{q}(X)=q(X)=0$ (cf.\,\cite[Corollary 4.18]{Deb01}).
Therefore, $$c(X,D)\le n+0-n=0.$$
Since $D$ is $f^{-1}$-invariant with $f$ being an int-amplified endomorphism, it follows from \cite[Theorem 1.4]{BH14} and \cite[Lemma 3.11]{Men20} that $(X,D)$ is a log canonical pair by noticing that the non-lc center of $(X,D)$ is $f^{-1}$-invariant and hence empty.
By the ramification divisor formula, since $f|_{X\backslash D}$ is \'etale, we have 
$$K_X+D=f^*(K_X+D).$$
So $K_X+D\equiv 0$ by \cite[Theorem 1.1]{Men20}.
Finally, applying Theorem \ref{thm-BMSZ2} to the pair $(X,D)$ with all the assumptions therein verified, we see that $(X,D)$ is a toric pair, and our theorem is thus proved.
\end{proof}

\section{Dynamics with Hodge index theorem}
We begin with the following type of Hodge index theorem which is known to experts.
\begin{lemma}[{cf.\,\cite[Corollarie 3.2]{DS04} or \cite[Lemma 3.2]{Zha16}}]
\label{lem_line_indep}
Let $X$ be a normal projective variety. 
Let $D_1\not\equiv0$ and $D_2\not\equiv 0$ be two nef $\mathbb{R}$-Cartier divisors such that $D_1\cdot D_2\equiv_w 0$. 
Then $D_1\equiv aD_2$ for some $a>0$.
\end{lemma}
\begin{proof}
We may assume $n:=\dim (X)\ge 2$.
Let $H$ be a very ample Cartier divisor on $X$. 
Let $S$ be a general surface on $X$ such that $H^{n-2}\equiv_w S$. 
Then $D_1|_S\cdot D_2|_S=D_1\cdot D_2\cdot H^{n-2}=0$. 
By the Hodge index theorem on $S$, 
we have $D_1|_S\equiv aD_2|_S$ for some $a>0$. 
Therefore,
$$(D_1-aD_2)\cdot H^{n-1}=(D_1-aD_2)^2\cdot H^{n-2}=0.$$
By \cite[Lemma 3.2]{Zha16}, $D_1\equiv aD_2$.
\end{proof}

We slightly generalize \cite[Claim 3.3]{Zha16} (also cf.\,\cite[Th\'eor\`eme 3.3]{DS04}) to the following, 
which states the semi-negativity of the generalized Hodge index theorem.

\begin{lemma}\label{lem_semi_pos}
Let $X$ be a normal projective variety of dimension $n$ and $M$ some $\R$-Cartier divisor.
Let $D_1,\cdots,D_{n-1}$ be nef $\mathbb{R}$-Cartier divisors such that $D_1\cdots D_{n-1}\not\equiv_w0$ and $M\cdot D_1\cdots D_{n-1}=0$. 
Then $M^2\cdot D_1\cdots D_{n-2}\le 0$.
\end{lemma}
\begin{proof}
We may assume $n\ge 2$.
Write $D_i=\lim\limits_{m\to\infty}D_{i,m}$ with $D_{i,m}$ ample $\mathbb{R}$-Cartier divisors for each $i$. 
Fix an ample Cartier divisor $H$ on $X$. 
Since $H\cdot D_{1,m}\cdots D_{n-1,m}>0$, 
we have
$$(M+r(m)H)\cdot D_{1,m}\cdots D_{n-1,m}=0$$
for some unique real number $r(m)$.	
Therefore,
$$(M+r(m)H)^2\cdot D_{1,m}\cdots D_{n-2,m}\le 0$$
by the negativity in \cite[Claim 3.3]{Zha16}. 
By the assumption $D_1\cdots D_{n-1}\not\equiv_w0$, we have $H\cdot D_{1,m}\cdots D_{n-1,m}>0$ and 
hence $\lim\limits_{m\to\infty}r(m)=0$.
Therefore,
$$M^2\cdot D_1\cdots D_{n-2}\le 0$$
by letting $m\to\infty$.
\end{proof}

The following lemma is known in \cite[Corollaire 3.5]{DS04} for the case of compact K\"ahler manifolds. We follow the idea there and reprove it in the algebraic context.
\begin{lemma}\label{lem_weak}
Let $X$ be a normal projective variety of dimension $n$.
Let $D,D',D_1,\cdots,D_k$ ($k\le n-2$) be nef $\mathbb{R}$-Cartier divisors such that $D\cdot D'\cdot D_1\cdots D_k\equiv_w 0$. 
Then $(aD+a'D')\cdot D_1\cdots D_k\equiv_w 0$ for some real numbers $(a,a')\neq (0,0)$.
Furthermore, if $D\cdot D_1\cdots D_k\not\equiv_w 0$, then $a'\neq 0$ and $(a,a')$ is unique up to a scalar.
\end{lemma}
\begin{proof}
We may assume $n\ge 2$.
If $D\cdot D_1\cdots D_k\equiv_w 0$,
we simply take $(a,a')=(1,0)$. 
In the following, we assume that $D\cdot D_1\cdots D_k\not\equiv_w 0$ (and hence $D_1\cdots D_k\not\equiv_w 0$). 

Fix ample Cartier divisors $A_1,\cdots,A_{n-k-1}$ on $X$. Denote by 
$$V:=\R D+\R D' \text{ and}$$
$$W:=\{x\in\textup{N}^1(X)~|~x\cdot D_1\cdots D_k\cdot A_1\cdots A_{n-k-1}=0\}$$
subspaces of $\N^1(X)$.
Note that 
$$D\not\in W \text{ and } \widetilde{D}:=aD+a'D'\in V\cap W$$ 
where
$a:=D'\cdot D_1\cdots D_k\cdot A_1\cdots A_{n-k-1}$
and $a':=-D\cdot D_1\cdots D_k\cdot A_1\cdots A_{n-k-1}\neq 0$.
Then $\dim_{\R} V\cap W\le 1$  and the uniqueness follows.

For each $1\le i\le n-k-1$, consider the following bilinear form on $\textup{N}^1(X)$:
$$q_i(x,y):=x\cdot y\cdot D_1\cdots D_k\cdot A_1\cdots A_{i-1}\cdot A_{i+1}\cdots A_{n-k-1}.$$
Then it follows from Lemma \ref{lem_semi_pos} and $D\cdot D'\cdot D_1\cdots D_k\equiv_w 0$ that $q_i$ is semi-negative on $W$ but semi-positive on $V$. Hence $q_i(\widetilde{D},\widetilde{D})=0$.

For any $w\in W$ and  $\lambda\in\mathbb{R}$, we  have $q_i(\lambda\widetilde{D}- w, \lambda\widetilde{D}-w)\le 0$. 
Then the inequality
$$q_i(w,w)-2\lambda q_i(\widetilde{D},w)\le 0$$
holds for any $\lambda\in\mathbb{R}$ and $w\in W$.
This happens only when $q_i(\widetilde{D},w)=0$ for any $w\in W$.

Note that $W$ and $A_i$ span $\N^1(X)$ because $W$ is a hyperplane of $\N^1(X)$ and $A_i\not\in W$.
Note also that $q_i(\widetilde{D}, A_i)=0$. 
Then 
$$q_i(\widetilde{D},x)=\widetilde{D}\cdot D_1\cdots D_k\cdot A_1\cdots A_{i-1}\cdot x\cdot A_{i+1}\cdots A_{n-k-1}=0$$
for any $x\in\N^1(X)$.
This implies that $V\cap W$ is independent of the choice of each $A_i$. So 
$$\widetilde{D}\cdot D_1\cdots D_k\cdot x_1\cdots x_{n-k-1}=0$$
for any divisors $x_1,\cdots, x_{n-k-1}\in \N^1(X)$, which means $\widetilde{D}\cdot D_1\cdots D_k\equiv_w 0$.
\end{proof}

\begin{proposition}[{cf.\,\cite[Lemme 4.4]{DS04}}]\label{pro_weak_eig}
Let $f:X\to X$ be a surjective endomorphism of a normal projective variety $X$ of dimension $n$. 
Let $D,D',D_1,\cdots, D_k$ ($k\le n-2$) be nef $\mathbb{R}$-Cartier divisors such that 
$D\cdot D_1\cdots D_k\not\equiv_w 0$
and 
$D'\cdot D_1\cdots D_k\not\equiv_w0$. 
Suppose $f^*(D\cdot D_1\cdots D_k)\equiv_w\lambda D\cdot D_1\cdots D_k$ and $f^*(D'\cdot D_1\cdots D_k)\equiv_w\lambda' D'\cdot D_1\cdots D_k$
for two real numbers $\lambda\neq \lambda'$. 
Then $D\cdot D'\cdot D_1\cdots D_k\not\equiv_w 0$.
\end{proposition}
\begin{proof}
Note that $n\ge 2$ since $\rho(X)\ge 2$ by the assumption.
Suppose the contrary that $D\cdot D'\cdot D_1\cdots D_k\equiv_w 0$. 
By Lemma \ref{lem_weak}, $(aD+a'D')\cdot D_1\cdots D_k\equiv_w 0$ for some $(a, a')\in \R^*\times \R^*$ unique up to a scalar. 
So for any $\mathbb{R}$-Cartier divisors $H_1,\cdots, H_{n-k-2}$, we have
$$(aD+a'D')\cdot D_1\cdots D_k\cdot H_1\cdots H_{n-k-2}=0.$$
Hence, by the projection formula, we have
$$(a\lambda D+a'\lambda 'D')\cdot D_1\cdots D_k\cdot f^*H_1\cdots f^*H_{n-k-2}=0.$$
Note that $f^*|_{\N^1(X)}$ is invertible. 
Then $(a\lambda D+a'\lambda' D')\cdot D_1\cdots D_k\equiv_w 0$, a contradiction with the uniqueness of $(a,a')$ up to a scalar and $\lambda\neq \lambda'$.	
\end{proof}

Now we state the main proposition in this section about the dynamical rigidity at first glance, which will be crucially used in Section \ref{Section_rc_prod}.
\begin{proposition}\label{prop-Di}
Let $f:X\to X$ be a surjective endomorphism of an $n$-dimensional normal $\mathbb{Q}$-Gorenstein projective variety $X$ of Fano type. 
Suppose that the eigenvalues of $f^*|_{\textup{N}^1(X)}$ (without counting multiplicities) are distinct real positive numbers $\{\lambda_1,\cdots,\lambda_r\}$ with $r\ge n$. 
Then the following hold.
\begin{enumerate}
\item $\rho(X)=n=r$.
\item $\Nef(X)$ is generated by base-point-free divisors $D_1,\cdots,D_n$ such that $f^*D_i\sim \lambda_i D_i$.
\item For each $i$, $\lambda_i$ is a positive integer, $D_i^2\equiv_w 0$ and $\kappa(X, D_i)=1$.
\item $D_1\cdots D_n>0$ and $\deg f=\lambda_1\cdots\lambda_n$.
\item $\Nef(X)=\PE^1(X)$.
In particular, $X$ is a Fano variety.
\end{enumerate}
\end{proposition}
\begin{proof}
It is trivial if $n=\dim(X)=1$.
Assume that $n\ge 2$.

By Lemma \ref{lem-mds}, $\Nef(X)$ is generated by base-point-free divisors $D_1,\cdots, D_m$. Note that $m\ge \rho(X)\ge r\ge n$.
Since $f^*|_{\N^1(X)}$ is linearly invertible and $\Nef(X)$ is $f^*|_{\N^1(X)}$-invariant,
$(f^s)^*$ fixes all the extremal rays $R_{D_1},\cdots, R_{D_m}$ of $\Nef(X)$ for some $s>0$.  
Therefore, by Lemma\,\ref{lem_positive_notreplaced}, we may assume that $f^*D_i\equiv \lambda_i D_i$ for $1\le i\le n$ and $f^*D_i\equiv \mu_i D_i$ for $n+1\le i\le m$ with $\mu_i\in\{\lambda_1,\cdots,\lambda_r\}$.
Since $D_i$ is integral, $\lambda_i$ is an integer.

We apply Proposition \ref{pro_weak_eig} several times.
Since $\lambda_1\neq \lambda_2$ and $\lambda_1\neq \lambda_3$, we have $D_1\cdot D_2\not\equiv_w0$ and $D_1\cdot D_3\not\equiv_w 0$.
Since $\lambda_1\lambda_2\neq\lambda_1\lambda_3$, we further have $D_1\cdot D_2\cdot D_3\not\equiv_w 0$. 
Repeating the same argument, we have
$D_1\cdot D_2\cdots D_{n-1}\cdot D_n\neq 0$
and $\deg f=\lambda_1\cdots\lambda_n$ by the projection formula.
So (4) is proved.

Suppose $m>n$. 
Note that $D_m$ and $D_{n}$ are nef and linearly independent in $\N^1(X)$.
Then $D_m\cdot D_n\not\equiv_w 0$ by Lemma \ref{lem_line_indep}. Similarly, $D_m\cdot D_{n-2}\not\equiv_w 0$. 
Hence  $D_m\cdot D_n\cdot D_{n-2}\not\equiv_w 0$ by Proposition \ref{pro_weak_eig}.
Repeatedly, we have $D_m\cdot D_{n}\cdot D_{n-2}\cdots D_1\neq 0$. 
Applying the projection formula, we have $\deg f=\mu_m\lambda_1\cdots\lambda_{n-2}\lambda_n$, which implies $\mu_m=\lambda_{n-1}$. 
By the same argument  after replacing $D_n$ by $D_{n-1}$, our
$D_m\cdot D_{n-1}\cdot D_{n-2}\cdots D_1\neq 0$ and thus  $\mu_m=\lambda_{n}=\lambda_{n-1}$, a contradiction.
In particular, $m=\rho(X)=r=n$. 

Replacing $D_m$ by $D_i$ in the above argument,
we have $D_i^2\equiv_w 0$ for each $i$.
Since $D_i$ is base point free, $\kappa(X, D_i)=1$.
Note that $\Pic_{\Q}(X) \cong \NS_{\Q}(X)$.
Then we have $f^*D_i\sim \lambda_i D_i$ after replacing $D_i$ by a suitable multiple.
So (1)--(3) are proved.

Let $D\in \partial \Nef(X)$.
Then without loss of generality, by (2), we may write
$D=\sum\limits_{i=1}^{n-1} a_i D_i$ with $a_i\ge 0$.
By (3), $D^n=0$ and thus $D$ is not big.
So $\Nef(X)=\PE^1(X)$.
Note that $-K_X$ is a big $\mathbb{Q}$-Cartier divisor.
Then $-K_X$ is further ample.
So (5) is proved.
\end{proof}

\section{Proof of Theorem \ref{thm_rc_product}}\label{Section_rc_prod}
In this section, we prove Theorem \ref{thm_rc_product} and use Notation \ref{not-5.1} throughout this section.
\begin{notation}\label{not-5.1}
Let $X$ be a smooth projective variety which is rationally connected of dimension $n$.
Let $$f:X\to X$$ be a surjective endomorphism such that the eigenvalues of $f^*|_{\N^1(X)}$ (without counting multiplicities) are $n$ distinct real numbers
$$\Lambda:=\{\lambda_1,\cdots,\lambda_n\}$$
which are greater than $1$.
\end{notation}

\begin{proposition}\label{prop-pencil}
There exist $f$-periodic prime divisors $D_1,\cdots, D_n$ such that:
\begin{enumerate}
\item each $D_i$ is a smooth Fano projective variety and $f^*D_i\sim \lambda_i D_i$;
\item there exist unique (up to isomorphism) $f$-equivariant fibrations 
$$\phi_i:X\to Y_i\cong \mathbb{P}^1$$ 
with $f|_{Y_i}$ being $\lambda_i$-polarized.
Each $D_i$ is a smooth fibre of $\phi_i$; and
\item replacing $f$ by a positive power, $(f|_{D_i})^*|_{\N^1(D_i)}$ has eigenvalues $\Lambda\backslash\{\lambda_i\}$.
\end{enumerate}
\end{proposition}
\begin{proof}
Since $f$ is int-amplified (cf.\,\cite[Theorem 1.1]{Men20}), $X$ is of Fano type by Theorem \ref{main-thm-y}. 
By Proposition \ref{prop-Di}, $X$ is Fano, $\rho(X)=n$, and there are $n$ base-point-free effective divisors $D_i\neq 0$ such that $f^*D_i\sim \lambda_i D_i$ and $\Nef(X)$ is generated by $D_1, \cdots, D_n$.

We denote by $$\phi_i:X\to Y_i$$  
the Iitaka fibration of $(X,D_i)$.
Then $D_i=\phi_i^*H_i$ for some ample $\mathbb{R}$-Cartier divisor $H_i$ on $Y_i$.
By Proposition \ref{prop-Di}, $\dim(Y_i)=\kappa(X, D_i)=1$.
Since $X$ is rationally connected, $Y_i\cong \mathbb{P}^1$.
For any curve $C$ with $\phi_i(C)$ being a point, 
by the projection formula,
we have 
$$D_i\cdot f_*C=f^*D_i\cdot C=\lambda_i D_i\cdot C=\lambda_i H_i\cdot (\phi_i)_*C=0.$$
Then $f(C)$ is also contracted by $\phi_i$. 
By the rigidity lemma (cf.\,\cite[Lemma 1.15]{Deb01}),  $\phi_i$ is $f$-equivariant and denote by $g_i:=f|_{Y_i}$. 
Note that $g_i$ is then $\lambda_i$-polarized.

Suppose $p_i:X\to Z_i\cong\mathbb{P}^1$ is another $f$-equivariant fibration such that $f|_{Z_i}$ is $\lambda_i$-polarized.
Let $F_i$ be the general fibre of $p_i$.
Then $p_i$ is the Iitaka fibration of $(X,F_i)$ and $f^*F_i\sim \lambda_i F_i$.
Then $F_i$ lies in the extremal ray $R_{D_i}$ and hence $p_i$ and $\phi_i$ are the same up to isomorphism.
We may replace $D_i$ by a general  fibre of $\phi_i$. 
Then $D_i$ is smooth and $f$-periodic.
So (2) is satisfied.

By the adjunction formula, $$-K_{D_i}=-(K_X+D_i)|_{D_i}\sim -K_X|_{D_i}$$
is ample.
So $D_i$ is Fano and (1) is satisfied.

By \cite[Theorem 5.1]{Fak03}, $g_i$ has Zariski dense periodic points.
So we may further replace $D_i$ by an $f$-periodic one.
After a suitable iteration of $f$, our $f|_{D_i}$ is a surjective endomorphism of $D_i$ for each $i$. 
By Proposition \ref{prop-Di}, $D_1\cdots D_n>0$. 
So $D_j|_{D_i}\not\equiv 0$ for $j\neq i$, and we have
$$(f|_{D_i})^*(D_j|_{D_i})\sim \lambda_j D_j|_{D_i}$$
for $j\neq i$.
So $(f|_{D_i})^*|_{\N^1(D_i)}$ has at least $n-1$ distinct real eigenvalues $\{\lambda_1,\cdots,\widehat{\lambda_i},\cdots,\lambda_n\}$.
Further all the eigenvalues of $(f|_{D_i})^*|_{\N^1(D_i)}$ are positive integers after replacing $f$ by a power, since $\textup{Nef}(D_i)$ is a rational polyhedron. So (3) is satisfied by applying Proposition \ref{prop-Di} for $D_i$.
\end{proof}


\begin{proposition}\label{prop-fano}
There are $f$-equivariant Fano contractions
$$\pi_i:X\to X_i,\, 1\le i\le n$$
of $K_X$-negative extremal rays, such that:
\begin{enumerate}
\item The eigenvalues of $(f|_{X_i})^*|_{\N^1(X_i)}$ are $\Lambda\backslash\{\lambda_i\}$.
\item $\pi_i$ is a conic bundle and $X_i$ is a smooth Fano   variety.
\end{enumerate}
\end{proposition}
\begin{proof}
We apply Proposition \ref{prop-pencil} and use the same notation there.
First note that $\sum\limits_{i=1}^n D_i$ is ample and 
$$(\sum\limits_{i=1}^n D_i)^n=(n!)D_1\cdots D_n>0.$$
So we have
$$D_1\cap\cdots\cap\widehat{D_i}\cap\cdots \cap D_n\neq\emptyset.$$
Let $C_i$ be an irreducible curve of $D_1\cap\cdots\cap\widehat{D_i}\cap\cdots \cap D_n$.
By Proposition \ref{prop-Di}, $X$ is Fano. Then we have
$$K_X\cdot C_i<0 \text{ and } D_j\cdot C_i=0$$ 
for $j\neq i$  (recall that $D_i^2\equiv_w 0$).
Since the space spanned by nef $D_1,\cdots,\widehat{D_i},\cdots, D_n$ is an $f^*$-invariant hyperplane of $\N^1(X)$, 
its dual space is $f_*$-invariant 1-dimensional and contains $R_{C_i}$ as an extremal ray in $\NE(X)$.
Therefore, $R_{C_i}$ induces an $f$-equivariant contraction (cf.\,\cite[Theorem 3.7]{KM98})
$$\pi_i:X\to X_i$$ 
for $1\le i\le n$.
By  Proposition \ref{prop-Di} (5), $\textup{Nef}(X)=\textup{PE}^1(X)$.
So it follows from Lemma \ref{lem-minimal} that $\pi_i$ is a Fano contraction, i.e., $\dim X_i<\dim X$.
By the cone theorem (cf.\,\cite[Theorem 3.7]{KM98}), 
for any $j\neq i$, $D_j=\pi_i^*L_j$ for some Cartier divisor $L_j$ on $X_i$.
Then $(f|_{X_i})^*L_j\equiv \lambda_j L_j$. 
By Proposition \ref{prop-Di}, $\rho(X)=n$ and hence
$$\rho(X_i)=\rho(X)-1=n-1.$$
So the  set of eigenvalues of $(f|_{X_i})^*|_{\N^1(X_i)}$ contains $\Lambda\backslash\{\lambda_i\}$ and hence coincides with it (cf.\,Proposition \ref{prop-Di}). 
This implies  (1).

Consider the morphism 
$$\psi: X\to X_i\times Y_i$$
induced from $\pi_i:X\to X_i$ and $\phi_i:X\to Y_i$ (cf.\,Proposition \ref{prop-pencil}).
Denote by $p_1: X_i\times Y_i\to X_i$ and $p_2:  X_i\times Y_i\to Y_i$ the two projections.
Let $$H:=p_1^*(\sum_{j\neq i}L_j)+p_2^*(\phi_i(D_i))$$
be a Cartier divisor on $X_i\times Y_i$.
Then $\psi^*H=\sum\limits_{i=1}^n D_i$ is ample and hence $\psi$ contracts no curve of $X$.
Note that $$\dim(X_i\times Y_i)=\dim(X_i)+1\le \dim(X).$$
So $\psi$ is a finite surjective morphism and we have 
$$\dim(X_i)=\dim(X)-1.$$
Now both $\psi$ and $p_1$ are equi-dimensional.
Then $\pi_i=p_1\circ \psi$ is also equi-dimensional.
By \cite[Theorem 3.1]{And85}, $\pi_i$ is a conic bundle and $X_i$ is smooth.
Note that $X_i$ is rationally connected.
Applying Proposition \ref{prop-Di} together with (1) for $X_i$, we further see that $X_i$ is a Fano variety. 
So (2) is proved.
\end{proof}

\begin{proof}[Proof of Theorem \ref{thm_rc_product}]
One direction is simple, for example we may take $$f=f_1\times\cdots\times f_n,$$ where $f_i([a:b])=[a^{i+1}:b^{i+1}]$.

Now we consider another direction and show by induction on $n=\dim(X)$. It is trivial if $n=1$.
If $n=2$, by Proposition \ref{prop-Di}, $X$ is a smooth Fano surface of Picard number 2 with $\Nef(X)=\PE^1(X)$.
Then $X\cong \mathbb{P}^1\times \mathbb{P}^1$.
From now on, suppose Theorem \ref{thm_rc_product} holds for $n-1$ with $n\ge 3$.

We use the same notation as in Propositions \ref{prop-pencil} and \ref{prop-fano}.
By induction, we have
$$X_i\cong D_i\cong (\mathbb{P}^1)^{\times (n-1)}.$$
Note that $f|_{X_i}$ splits (cf.\,Remark \ref{rmk-split}) and for $j\neq i$,
there are $f$-equivariant natural projections 
$$p_j:X_i\cong (\mathbb{P}^1)^{\times (n-1)}\to Z_j\cong\mathbb{P}^1$$
such that $f|_{Z_j}$ is $\lambda_j$-polarized.
We may assume $Z_j=Y_j$ and $p_j\circ \pi_i=\phi_j$ by the uniqueness property in Proposition \ref{prop-pencil}.
Then $D_1\cap\cdots\cap\widehat{D_i}\cap\cdots\cap D_n$ intersects transversally and is a general fibre $\ell_i$ of $\pi_i$.
In particular, we have
$$d:=D_1\cdots D_n=D_i\cdot \ell_i$$
and $D_i\cdot \ell_j=0$ for $i\neq j$ since $D_i^2\equiv_w 0$.
Note that $K_X\cdot \ell_i=-2$.
Since $D_1,\cdots,D_n$ is a basis for $\N^1(X)$, we may write 
$$K_X\equiv\sum_{i=1}^na_iD_i$$
for some rational numbers $a_i$. 
Intersecting the above numerical equivalence with $\ell_i$, we have $a_i=-2/d$ for each $i$ and thus 
$$-dK_X\equiv 2\sum_{i=1}^nD_i.$$
Therefore,
$$(-dK_X)^{n-1}\cdot D_1= 2^{n-1}(n-1)!d.$$
From another aspect, we apply the adjunction formula
$$K_{D_1}=(K_X+D_1)|_{D_1}=K_X|_{D_1}$$ 
and note that
$$(-K_{D_1})^{n-1}=2^{n-1}(n-1)!$$
since $D_1\cong (\mathbb{P}^1)^{\times (n-1)}$.
Then 
$$(-dK_X)^{n-1}\cdot D_1=(-dK_{D_1})^{n-1}=2^{n-1}(n-1)!d^{n-1}$$
Therefore, $d^{n-1}=d$ and hence $d=1$ since we assumed $n\ge 3$.

Consider the morphism 
$$\psi: X\to X_1\times Y_1\cong (\mathbb{P}^1)^{\times n}$$
induced from $\pi_1:X\to X_1$ and $\phi_1:X\to Y_1$.
Note that 
$$\deg \psi=D_1\cdot \ell_1=d=1.$$
So the theorem is proved.
\end{proof}

\end{document}